\documentclass[12pt,a4paper]{article}
\usepackage{float}
\usepackage[english]{babel}
\usepackage{latexsym}
\usepackage {amssymb}
\usepackage {amsmath}
\usepackage {mathtools}
\usepackage{color}
\usepackage{hyperref}
\usepackage{setspace}
\usepackage{booktabs}
\usepackage[]{algorithm2e}
\usepackage{multirow}
\usepackage[titletoc,title]{appendix}

\usepackage{amsthm, amssymb,amsmath, amsfonts,bm,xcolor,hyperref}
 
\theoremstyle{plain}

\usepackage{rotating}

\usepackage[numbers]{natbib}

\def\beqn{\begin{eqnarray*}}
\def\eeqn{\end{eqnarray*}}
\def\beq{\begin{eqnarray}}
\def\eeq{\end{eqnarray}}

\newtheorem{theorem}{Theorem}

\newtheorem{corollary}{Corollary}
\newtheorem{remark}{Remark}

\makeatother
\makeatletter
\renewcommand{\@fnsymbol}[1]{\@arabic{#1}}
\makeatother
\title{Robust oracle estimation and uncertainty quantification 
for possibly sparse quantiles} 

\author{
Eduard Belitser  
\footnote{
Department of Mathematics, VU Amsterdam, The Netherlands}  
\and
Paulo Serra
\footnote{
Department of Mathematics, VU Amsterdam, The Netherlands.} 
\and
Alexandra Vegelien
\footnote{
Department of Mathematics, VU Amsterdam, The Netherlands.} 
}

\begin{document}

\baselineskip=25pt

\maketitle

\begin{abstract}
\baselineskip=15pt
\noindent
A general {\it many quantiles + noise} model is studied in the robust formulation 
(allowing non-normal, non-independent observations), where the 
identifiability requirement for the noise is formulated in terms of quantiles 
rather than the traditional zero expectation assumption. We propose a penalization method 
based on the quantile loss function with appropriately chosen penalty function making inference 
on possibly sparse high-dimensional quantile vector. We apply a local approach to address the optimality by 
comparing procedures to the oracle sparsity structure. We establish that the proposed procedure 
mimics the oracle in the problems of estimation and uncertainty quantification 
(under the so called \emph{EBR condition}). Adaptive minimax results over sparsity scale 
follow from our local results. 

\end{abstract}

\newpage

\section{Introduction}

Since the emergence of statistical science, the classical `signal+noise' paradigm for
observed data is the main testbed in theoretical statistics for a vast number of methods 
and techniques for various inference problems and in various optimality frameworks. 
Many other practically important situations can be reduced to or approximated by the 
`signal+noise' setting, capturing the statistical essence of the original (typically, more complex)
model and preserving its main features in a pure form. 
There is a huge literature on `signal+noise' setting with big variety of combinations of 
the main ingredients in the study. We mention the following ingredients: assumptions on 
the observation model (moment conditions, independence, normality, etc.); applied methodology 
(LSE, penalization, shrinkage, thresholding, (empirical) Bayesian approach, projection, FDR method); 
studied inference problem (estimation, detection, model selection, testing, posterior contraction, 
uncertainty quantification, structure recovery); structural assumptions (smoothness, sparsity, 
clustering, shape constraints such as monotonicity, unimodality and convexity, loss functions 
(e.g., expectations of powers of $\ell_2$-norm, $\ell_q$-norm; Hamming loss, etc.); 
optimality framework and pursued criteria (minimax rate, oracle rate, control of FDR, FNR, 
expectation of the loss function, exponential bounds for the large deviations of the loss functions, etc.). 
A small sample of relevant literature includes 
 \cite{Donoho&etal:1992}, \cite{Donoho&Johnstone:1994}, \cite{Benjamini&Hochberg:1995}, 
 \cite{Birge&Massart:2001}, \cite{Johnstone&Silverman:2004},  
 \cite{Baraud:2004}, \cite{Abramovich&etal:2006}, \cite{Efron:2008}, \cite{Babenko&Belitser:2010}, 
 \cite{Belloni&etal:2011}, \cite{Castillo&vanderVaart:2012}, 
 \cite{Martin&Walker:2014}, \cite{Johnstone:2017}, \cite{Belitser:2017}, \cite{Bellec:2018}, \cite{Butucea&etal:2018}, \cite{Belitser&Nurushev:2020}. 

Suppose we observe $X=(X_1,\ldots, X_n)\sim 
\mathbb{P}_\theta^{n}$, where $\theta=(\theta_1,\ldots, \theta_n)\in\mathbb{R}^n$ 
is the parameter of interest. 
We shall suppress $n$ from the notation to simply write $\mathbb{P}_\theta$ 
for $\mathbb{P}_\theta^n$. A natural example is given by 
$X_i\overset{\rm ind}{\sim} \mathrm{N} (\theta_i,1)$, $i=1,\ldots,n$, 
but this specific distributional assumption on the data generating process will not be imposed.
the main modeling assumption in this note is that, for some fixed level $\tau\in(0,1)$, 
\[
\mathbb{P}(X_i-\theta_i \le 0)=\tau, \quad i=1,\dots,n.
\] 
In other words, we deal with the \emph{many quantiles model}, as $\theta_i$'s 
are the $\tau$-quantiles of the observed $X_i$'s and the goal is to make an inference 
on the high-dimensional parameter $\theta$.
One can study and make inference on several quantiles simultaneously across different levels $\tau\in(0,1)$.
Looking at quantiles instead of expected values can be a more sensible object of study, 
especially when interest focuses on the effect of the parameter values on the tails of a distribution 
of observations, in addition to, or instead of, the center. It is also well known that modeling of 
the median as opposed to the mean is much more robust to outlying observations; cf.\ \cite{Huber:2011}. 
In `signal+noise' models with iid noise, when estimates for quantiles of different levels are combined, their asymptotic efficiency relative to least squares can be arbitrarily large if the noise is, for instance, a two-component mixture of Gaussians; cf.\ \cite{Zou&Yuan:2008}, \cite{Jiang&etal:2013}, \cite{Jiang&etal:2014}.
Quantiles as objects of interest occur in diverse areas, including economics,
biology, meteorology; the values at risk in finance has the meaning of quantiles; cf.\  \cite{Santosa&Symes:1986}, \cite{Abrevaya:2002}, \cite{Koenker:2005},  \cite{Hao&Naiman:2007}, 
\cite{Belloni&Chernozhukov:2011}, \cite{Hulman&etal:2015}, \cite{Ciuperca:2018}, \cite{Belloni&etal:2019}.
In this paper, we work with a new setting `many quantiles + noise',  and therefore use 
the quantile loss rather than $\ell_2$-loss. The quantile loss function has certain properties of usual loss functions, but it is not symmetric, in fact this is the main peculiarity essentially characterizing 
the notion of quantile.

Besides quantile formulation, the next distinctive feauture of our study is 
the \emph{robust formulation} in the sense that we do not assume 
any particular form of the distribution of the observed data. Only a mild condition (Condition C1) 
is imposed. The observations do not need to  
be normal and, in fact no specific distribution is assumed,
the distribution of the `noise' terms $X_i-\theta_i$, $i=1,\ldots, n$, 
may depend on $\theta$ and do not even have to be independent. 
This makes the applicability scope of our results rather broad. 

The third aspect of our study is the {\emph{local approach}. It is in general impossible to make 
sensible inference on a hight-dimensional parameter without 
any additional structure, either in terms of assumptions on parameter, 
or on the observation model, or both. In this paper, we are concerned with 
\emph{sparsity structure}, various version of which have been predominantly studied 
in the literature recently. Sparsity structure means that a relative majority of the parameter 
entries are all equal to some fixed value (typically zero). 
In the \emph{local approach} we address optimality by comparing procedures 
to the \emph{oracle sparsity structure} (think of an `oracle observer' that knows 
the best sparsity pattern of the actual $\theta$). 
When applying  the local approach, the idea is not 
to rely on sparsity as such, but rather extract all the sparsity structure present 
in the data, and utilize it in the aimed inference problems. 
The claim that a procedure performs optimally in the local sense means that it   
attains the oracle quality, i.e., mimics the best sparsity structure pertinent to 
the true parameter, whichever it is. The local result imply minimax optimality over 
al scales (including the traditional sparsity class $\ell_0[s]$) that satisfy a certain condition.  

The final feature of our study concerns the problem of \emph{uncertainty quantification}.
In recent years, focus in non-parametric statistics has shifted from point estimation to 
uncertainty quantification, a more challenging problem, much less results on this topic  
are available in the literature; cf.\ \cite{Szabo&etal:2015}, \cite{Belitser:2017}, 
\cite{vanderPas&etal:2017}, \cite{Belitser&Nurushev:2020}.
Certain negative results (cf.\ \cite{Li:1989}, \cite{Baraud:2004}, \cite{Belitser&Nurushev:2019} and the  
references therein) show that in general a confidence set cannot simultaneously have coverage 
and the optimal size uniformly over all parameter values. 
This makes the construction of the so called `honest' confidence sets impossible, 
and a strategy recently pursued in the literature is to discard a set of `deceptive parameters' 
to ensure coverage at the remaining parameter values, while maintaining the optimal size 
uniformly over the whole set. 
In an increasing order of generality, removing deceptive parameters is expressed by imposing 
the conditions of \emph{self-similarity}, \emph{polished tail} and \emph{excessive bias restriction} (EBR), see
\cite{Szabo&etal:2015}, \cite{Belitser:2017}, \cite{vanderPas&etal:2017}, \cite{Belitser&Nurushev:2020}.
All the above papers deal with $\ell_2$-loss framework for Gaussian observations in 
`many means + noise' setting (a more general case is considered in \cite{Belitser&Nurushev:2020}). 

In this paper, we work with a new setting `many quantiles + general noise',  and use 
the quantile loss rather then $\ell_2$-loss, so we propose and exploit a new version of EBR condition 
expressed in terms of quantile loss. To the best of our knowledge, there are no results on uncertainty 
quantification neither for the quantile loss function, nor for the $\ell_1$-loss, also neither in local, nor global 
(minimax) formulation in this general robust setting. In this respect, our study presents the first step in this direction.

We finally summarize the main contributions of this paper.
\begin{itemize}
\item
For our new (robust) setting `many quantiles + general noise' 
we propose a \emph{penalization procedure} for selecting the sparsity pattern, 
based on the quantile loss (as counterpart of $\ell_2$ approximation term) 
and the penalty term corresponding to the quantile loss (as counterpart of model complexity term).  
\item
The estimated sparsity pattern is next used for the construction of an
estimator of $\theta$ and a confidence set for $\theta$, 
again in terms of the quantile loss function. 

\item 
Two theorems establish the local (oracle) optimality of the estimator and confidence 
set (under the EBR condition), respectively. 
\item
We provide an elegant and relatively short proofs of the main results; compare with the relatively 
laborious proofs of related results in \cite{Szabo&etal:2015}, \cite{Belitser:2017}, 
\cite{vanderPas&etal:2017}, \cite{Belitser&Nurushev:2020}.
\item
The obtained results are robust and local in the sense explained above.
\item
The obtained local results imply adaptive minimax optimality for scales of classes that satisfy certain relation 
(in particular, for the traditional sparsity scale).   

\end{itemize}

Organization of the rest of the paper is as follows.
In Section \ref{section_preliminaries} we give a robust formulation of the observation model 
in the `many quantiles + general noise' setting, and introduce some notations and preliminaries. 
In Section \ref{section_main} we  present the main results of the paper and discuss 
their consequences for the optimality in the minimax sense for the traditional sparsity scale. 
The proofs of the theorems are provided in Section \ref{section_proofs}. 
\section{Robust model formulation and preliminaries}
\label{section_preliminaries}

We can formally rewrite the model stated in the introduction in the familiar `signal+noise' form:
\begin{align}
\label{model}
X_i = \theta_i +\xi_i, \quad \mathbb{P}_\theta (\xi_i \le 0)=\tau, \;\;
i=1,\ldots, n.
\end{align}
where $\xi_i = X_i-\theta_i$ can be thought of as `noise'.  
Note that all the quantities depend of course on $\tau$. Precisely, for $\tau\in(0,1)$, 
in the model \eqref{model} we have $\theta=\theta(\tau)$, $\xi=\xi(\tau)$ with 
$\mathbb{P}_\theta(\xi_i(\tau)\le 0)= \tau$, $i=1,\ldots,n$. Not to overload the notation, we suppress 
the dependence on $\tau$ in the sequel. 
As we mentioned in the introduction, we study robust setting in the sense that we do not assume 
any particular distribution of $\xi=(\xi_1,\ldots,\xi_n)$. In fact, the $\xi_i$'s do not have to be 
identically distributed, their distribution may depend on $\theta$ and  they do not even have to 
be independent. 

Introduce the following asymmetric absolute deviation function:  
\[
\rho(x)=\rho_\tau(x)=\rho_{\tau,m}(x)=\sum_{i=1}^m x_i (\tau-\mathrm{1}\{x_i \le 0\}), 
\quad x\in\mathbb{R}^m, \;\; m\in\mathbb{N},
\] 
which we will also call \emph{quantile loss function}.
 
Slightly abusing notation, we use the same notation $\rho(x)$ for any $m\in\mathbb{N}$, 
for example, most of the time it will be either $m=n$ or $m=1$, depending on the dimensionality of 
the argument of the function. This will always be clear from the context. 
Often, we will suppress the dependence of $\rho$ on a fixed $\tau \in(0,1)$, unless emphasizing 
this dependence where it is relevant.  
\begin{remark}
The origin of the quantile loss function $\rho_\tau(x)$, $x\in\mathbb{R}$, is well explained in the book 
\cite{Koenker:2005}. The function $\rho_\tau(x)$ characterizes the $\tau$-quantile of a random 
variable $Z$ in the sense that the $\tau$-quantile $\theta_\tau$ of $Z$ is known 
to minimize the criterion function $\mathbb{E}\rho_\tau(Z-\vartheta)$ with respect to 
$\vartheta$, i.e., $\arg\min_{\vartheta\in\mathbb{R}} \mathbb{E}\rho_\tau(Z-\vartheta) = \theta_\tau$. 
Basically, this function plays the same role for characterizing the quantiles as the quadratic 
function in characterizing the expectation of a random variable. 

The quantile loss function $\rho$ can be related to  the $\ell_1$-criterion 
$|x|=\sum_{i=1}^m |x_i|$, 
$x=(x_1,\ldots, x_m) \in \mathbb{R}^m$: for any $\tau\in(0,1)$,
\begin{align}
\label{prop_rho_abs}
(1-c_\tau) |x|\le\rho_\tau(x)\le c_\tau |x|, \;\; 
\text{with } \;\; c_\tau= \max\{\tau,1-\tau\}.
\end{align}
In particular, it follows that $\rho_{1/2}(x) = |x|/2$.
Another useful property of the quantile loss function $\rho$ to be used later on is that, 
for any $x\in \mathbb{R}^m$,
\begin{align}
\label{prop_-x}
\rho_\tau(-x) \le C_\tau \rho_\tau(x),  \quad \text{with} \quad 
C_\tau=\max\{\tfrac{1-\tau}{\tau}, \tfrac{\tau}{1-\tau}\} =\tfrac{c_\tau}{1-c_\tau}.
\end{align}
\end{remark}

\begin{remark}
\label{rem2}
The above quantile loss function evaluated at the difference 
$\rho(x-y)$, $x,y \in \mathbb{R}^m$, 
possesses all the properties of a metric $d(x,y)$ except for the symmetry. 
In particular, $\rho(0) = 0$ (zero at zero); if $x\not=y$, $\rho(x-y)> 0$ (positivity); and finally 
the triangle inequality holds
\begin{align}
\label{triangle_ineq}
\rho(x+y) \le \rho(x) + \rho(y), \quad x,y\in \mathbb{R}^m.
\end{align}
\end{remark}
    
Since we have as many parameters as observations, it is typically impossible to make inference on  
$\theta$ even in a weak sense unless the data possesses some structure.
Here, we work with the structural assumption that $\theta$ is (possibly) a sparse vector.
Specifically, we assume that $\theta\in L_I$, where 
$I \subseteq [n]=\{1,...,n\}$ and $L_I = \{x \in \mathbb{R}^n : x_i = 0,  i \in I^c \}$, where 
$I^c=[n]\setminus I$. The structure here 
is the unknown `true' sparsity pattern $I_0=I_0(\theta)$, that is, $\theta \in L_{I_0}$ and 
$I_0$ is the `minimal' sparsity structure in the sense that 
$I_0=\arg\min_{I\in\mathcal{I}: L_I \ni \theta} \sum_{i\in I} i$.  

Introduce some further notation:  for $x\in\mathbb{R}^n$, denote its $\ell_1$-norm by
$|x|=\sum_{i=1}^n |x_i|$, $|S|$ denote the cardinality of a set $S$. 
Further let us introduce the so called `quantile projection' $\mathrm{P}_{I}$ 
onto $L_I$ (called just projection in what follows), with respect to the quantile loss 
function $\rho(x-y)$. For $x \in \mathbb{R}^n$, define $\mathrm{P}_{I} x$ as follows: 
\[
\inf_{y\in L_I} \rho(x-y)=\rho(x-\mathrm{P}_{I} x).
\]
In view of the properties of the quantile loss function $\rho$ given by Remark \ref{rem2}, 
the projection $\mathrm{P}_I$ is readily found as the following linear operator:
\[
\mathrm{P}_I x=(x_i\, 1\{i\in I\},\, i=1,\ldots,n)\in \mathbb{R}^n.
\]
We will use later a certain monotonicity property of the quantile projection operator.
\begin{align}
\label{prop_monoton}
\text{If} \;\; I_1\subseteq I_2, \;\;\text{then} \quad  \rho(\mathrm{P}_{I_1} x) \le \rho(\mathrm{P}_{I_2} x), \;\;
x \in\mathbb{R}^n.
\end{align} 

Denote $\lambda(s)=s \log (en/s)$ (with the convention $0\log (a/0)=0$ for $a>0$) and 
\[
p(I)=\big(|I| \lambda(|I|)\big)^{1/2}= |I|\big[\log(en/|I|)\big]^{1/2}.
\] 
We have that $\lambda(s) \ge s$ for $s\in[0,n]$, and 
besides, since $\lambda(s)$ is increasing in $s\in(0,n]$, $\lambda(s) \ge 
1 + \log(n)$ for all $s\in [n]$.

The following relation will be used later: for any $\nu>1$,
\begin{align}
\label{rel_nu}
\sum_{I\in\mathcal{I}}e^{-\nu\lambda(|I|)}= \sum_{s\in[n]}
\sum_{I: |I|=s} e^{-\nu\lambda(s)}\le \sum_{s\in[n]} e^{-(\nu-1)s} 
\le C_\nu,
\end{align}
with $C_\nu=(e^{\nu-1}-1)^{-1}<\infty$.

The following conditions on $\xi=X-\theta$ is assumed throughout.\smallskip\\   
{\sc Condition C1.}
For some $M_\xi, \alpha_\xi, H_\xi >0$ and all $M\ge 0$ and all $I \in \mathcal{I}$,    
\begin{align}
\sup_{\theta\in\mathbb{R}^n} 
\mathbb{P}_\theta\big( \rho(\mathrm{P}_I \xi)> M_\xi p(I) +|I|^{1/2}M^{1/2}\big) \le
H_\xi e^{-\alpha_\xi M}.
\label{C1}
\end{align}
Let $\mathcal{P}_\alpha$ stand for the family of distributions 
$\mathbb{P}_\theta$ satisfying Condition C1.  
Notice that \eqref{C1} is always fulfilled for $I=\varnothing$. Then for any $I\in\mathcal{I}$ with 
$|I|\ge 1$, and any $M>M_\xi$,
\[
\mathbb{P}_\theta\big( \rho(\mathrm{P}_I \xi)> M p(I)\big) \le
H_\xi e^{-\alpha_\xi (M-M_\xi) \lambda(|I|)} \le e^{-\alpha_\xi (M-M_\xi) \lambda(1)}.
\]

\begin{remark}
Condition C1 is mild, for instance, it holds for independent  {\it sub-gaussian} $\xi_i$'s.
Recall one of the equivalent definitions of sub-gaussianity: a random variable 
$W$ is called $\sigma$-{\it subgaussian} for $\sigma>0$ if for some $c_0>0$ 
$\mathbb{E}e^{c_0 W^2/\sigma^2} \le 2$.
In our case, $\sigma=1$. Recall that $\|x\|^2=\sum_{i=1}^n x_i^2 $ denotes 
the usual (squared) $\ell_2$-norm of $x\in\mathbb{R}^n$.
Now, since $\mathrm{P}_I \xi$ has at most $|I|$ non-zero entries, by \eqref{prop_rho_abs} and 
the Cauchy-Schwartz inequality we have 
$\rho(\mathrm{P}_I \xi)\le c_\tau |\mathrm{P}_I \xi| \le c_\tau|I|^{1/2} \|\mathrm{P}_I \xi \|$. 
Recall also that $\lambda(s) \ge s$, $s\in[0,n]$.
Using these and the Markov inequality, we obtain that, for $M_\xi \ge c_\tau (c_0^{-1}\log 2)^{1/2}$,  
\begin{align*}
\mathbb{P}\big(\rho(\mathrm{P}_I \xi) &\ge M_\xi p(I)+|I|^{1/2}M^{1/2}\big)\\
&\le \mathrm{P}\Big(c_\tau |I|^{1/2} \|\mathrm{P}_I \xi \|_2  
\ge M_\xi p(I)+|I|^{1/2}M^{1/2} \Big)\\
& =\mathrm{P}\Big(\|\mathrm{P}_I \xi \|_2 \ge \tfrac{M_\xi}{c_\tau} [\lambda(|I|)]^{1/2}
+c_\tau^{-1}M^{1/2} \Big)\\
&\le \mathrm{P}\big(c_0 \|\mathrm{P}_I \xi \|_2^2\ge c_0 c_\tau^{-2} M_\xi^2  \lambda(|I|)
+ c_0 c_\tau^{-2} M\big)\\
&\le e^{|I| \log 2 - c_0M_\xi^2 q^{-2}_\tau\lambda(|I|)-c_0 c_\tau^{-2} M}  
\le e^{-c_0 c_\tau^{-2} M}.
\end{align*}
\end{remark}

\begin{remark}
We can extend our setting by including also a parameter $\sigma>0$ by considering 
$\sigma\xi_i$ instead of just $\xi_i$ (and $(X_i-\theta)/\sigma$ instead of $X_i-\theta$) 
in \eqref{model} and in Condition C1. 
In that case, the parameter $\sigma>0$ is assumed to be known 
and fixed throughout. Together with $n$, it reflects the information amount in the model in the sense that 
$\sigma \to 0$ means the flux of information.  
\end{remark}

We propose a penalized projection estimator. For $\kappa>0$, define 
$\hat{I}=\hat{I}(X,\kappa)$ to be a minimizer of the criterion $C(I)=C(I,\kappa)=C(I,\kappa,X)$:  
\begin{align}
C(I)&=\rho(X-\mathrm{P}_I X)+\kappa p(I) =\rho(\mathrm{P}_{I^c} X)+\kappa p(I), \notag\\ 
\label{def_crit}
\min_{I\in\mathcal{I}}C(I)&=C(\hat{I}). 
\end{align}
\begin{remark}
Since the proposed procedure $\hat{I}$ for selecting sparsity pattern is based on the 
non-symmetric quantile loss function $\rho$, the deviations from below and from above 
are treated differently. But the computation of the estimator $\hat{I}$ of the sparsity pattern 
is not difficult. Indeed, it can be reduced to the search over $n$ options: 
with $b_k= \rho(X_k) \ge 0$, $k\in[n]$,
\begin{align*}
\min_{I \in \mathcal{I}} C(I) 
&= \min_{I \in \mathcal{I}}\{\rho(\mathrm{P}_{I^c} X)+\kappa p(I)\} 
= \min_{I \in \mathcal{I}} \Big\{\sum_{k\in I^c} \rho(X_k)+\kappa p(I)\Big\} \\
&= \min_{i\in [n]} \Big\{\min_{I\in\mathcal{I}: |I|=i}
\sum_{k\in I^c} \rho(X_k)+\kappa i\log^{1/2}(\tfrac{en}{i})\Big\}  \\
&=\min_{i\in [n]} \Big\{\sum_{k=1}^{n-i} b_{(k)} + \kappa (i\log^{1/2}(\tfrac{en}{i})\Big\}
=\min_{i\in [n]} \big\{B_i+ \kappa (i\log^{1/2}(\tfrac{en}{i})\big\},
\end{align*}
where $b_{(1)} \le \ldots  \le b_{(n)}$ is the ordered 
sequence of $b_k$'s and $B_i =\sum_{k=1}^{n-i} b_{(k)}$.
\end{remark}

Now, using the estimated sparsity structure $\hat{I}$, define the estimator 
\begin{align}
\label{def_estimator}
\hat{\theta}=\hat{\theta}(X,K)=\mathrm{P}_{\hat{I}} X.
\end{align}

From the triangle inequality \eqref{triangle_ineq}, it follows that $\rho(x-y) \ge\rho(x) -\rho(y)$.
Using this, \eqref{prop_-x} and the definition \eqref{def_crit}, we derive that, for any $I'\in\mathcal{I}$, 
\begin{align}
\kappa(p(\hat{I})-p(I')) &\le \rho(\mathrm{P}_{I'^c} X)- \rho(\mathrm{P}_{\hat{I}^c} X)=
\rho(\mathrm{P}_{\hat{I} \backslash I'} X) - \rho(\mathrm{P}_{I' \backslash \hat{I}}X) \notag \\ 
& \le \rho(\mathrm{P}_{\hat{I} \backslash I'} \theta) + \rho(\mathrm{P}_{\hat{I} \backslash I'} \xi) 
- \rho(\mathrm{P}_{I' \backslash \hat{I}} \theta) + \rho(\mathrm{P}_{I' \backslash \hat{I}} (-\xi)) \notag \\
&\le \rho(\mathrm{P}_{I'^c} \theta)- \rho(\mathrm{P}_{\hat{I}^c} \theta) + 
 \rho(\mathrm{P}_{\hat{I}} \xi) + C_\tau \rho(\mathrm{P}_{I'}\xi).
\label{def_hat_I}
\end{align}

For $\varkappa>0$, $\theta\in\mathbb{R}^n$ and $I \in\mathcal{I}$, introduce the quantity
\[
r(\theta, I)=r_\varkappa(\theta, I)=\rho(\theta-\mathrm{P}_I(\theta))+\varkappa p(I), 
\] 
which we call the {\it quantile rate} of the sparsity structure $I$.
The oracle sparsity structure $I_o=I_o(\theta)=I_o(\theta,\varkappa)$ is the one minimizing  $r_\varkappa(\theta, I)$:
\begin{align}
\label{def_oracle}
\min_{I\in\mathcal{I}} r_\varkappa(\theta,I)= r_\varkappa(\theta,I_o) = r_\varkappa(\theta)=r(\theta).
\end{align}
where its minimal value $r(\theta)$ is called the {\it oracle quantile rate}.

\section{Main results}
\label{section_main}
In this section we give the main results. All the constants in the below 
assertions depend on the fixed constants $M_\xi, H_\xi, \alpha_\xi$  appearing in Condition C1, 
the constant $\kappa$ appearing in the oracle structure definition \eqref{def_crit}, and the quantile level $\tau$. 
 
The next result concerns the estimation problem.
\begin{theorem}[Estimation]
\label{th1}
Let Condition C1 be fulfilled. Then for any $\varkappa>0$ and sufficiently large $\kappa$, 
there exist positive $M_1, H_1, m_1$ such that
\begin{align}
\label{th1_a}
\mathbb{P}_\theta\big(\rho(\theta-\hat{\theta})\ge M_1 r(\theta)\big) \le 
H_1 e^{-m_1 \lambda(1)} \le
H_1 n^{-m_1},
\quad \theta\in\mathbb{R}^n,
\end{align}
where $\hat{\theta}$ is defined by \eqref{def_estimator}.
\end{theorem}

Next we address the new problem of \emph{uncertainty quantification} (UQ) 
for the parameter $\theta$. A confidence set is defined in terms of quantile loss function $\rho$ as follows:
\begin{align}
\label{def_conf_set}
B(\hat{\theta}_0,\hat{r}_0)=\{\theta \in\mathbb{R}^n: \rho(\theta-\hat{\theta}_0) \le \hat{r}_0\},
\end{align}
where the `center' $\hat{\theta}_0=\hat{\theta}_0(X):\mathbb{R}^n \mapsto \mathbb{R}^n$ and `radius' 
$\hat{r}_0=\hat{r}_0(X): \mathbb{R}^n \mapsto \mathbb{R}_+ =[0,+\infty]$ are
measurable functions of the data $X$. The goal is to construct such a confidence set 
$B(\hat{\theta},C\hat{r})$ that for any $\alpha_1,\alpha_2\in(0,1]$ and some function 
$R(\theta)$, $R:\mathbb{R}^n\rightarrow \mathbb{R}_+$, there exist $C, c > 0$ such that
\begin{align}
\label{uq_framework}
\sup_{\theta\in\Theta_{\rm cov}}\mathbb{P}_\theta\big(\theta\notin 
B(\hat{\theta},C\hat{r})\big)\le\alpha_1, \quad
\sup_{\theta\in\Theta_{\rm size}}\mathbb{P}_\theta\big(\hat{r}\ge 
c R(\theta)\big)\le\alpha_2,
\end{align}
for some $\Theta_{\rm cov}, \Theta_{\rm size} \subseteq \mathbb{R}^n$. 
The function $R(\theta)$, called \emph{radial rate}, is a benchmark 
for the effective radius of the confidence set $B(\hat{\theta},C\hat{r})$. 
The first expression in \eqref{uq_framework} is called \emph{coverage relation} 
and the second \emph{size relation}. It is desirable to find the smallest 
$r(\theta)$, the biggest  $\Theta_{\rm cov}$ and $\Theta_{\rm size}$ such 
that \eqref{uq_framework} holds and $R(\theta)\asymp r(\Theta_{\rm size})$, 
where $r(\Theta_{\rm size})$ is the optimal rate 
in estimation problem for $\theta$. In our local approach, we pursue even more ambitious goal 
$R(\theta)\asymp r(\theta)$, where $r(\theta)$ is the oracle rate from the (local) estimation problem. 

Typically, the so called \emph{deceptiveness} issue arises for the UQ problem in that 
the confidence set of the optimal size and high coverage can only be constructed 
for non-deceptive parameters (in particular, $\Theta_{\rm cov}$ cannot be the whole 
set $\mathbb{R}^n$). That is, coverage with an optimal sized radius is only possible 
if certain deceptive set of parameters is excluded from consideration, which 
is expressed by imposing some condition on the parameter.
For example, the EBR (excessive bias restriction) condition in case of $\ell_2$-norm 
is proposed in \cite{Belitser:2017}--\cite{Belitser&Nurushev:2020}. 
Here we need an EBR-like condition (we keep the same term \emph{EBR}), but now in terms of 
the quantile loss function $\rho$. \smallskip\\
{\sc Condition EBR.}
We say that $\theta\in\mathbb{R}^n$ satisfies the \emph{excessive bias restriction} (EBR) condition 
with structural parameter $t\ge 0$ if 
\[
\rho(\theta-\mathrm{P}_{I_o}\theta)\le t p(I_o),
\]
where the oracle structure $I_o=I_o(\theta)$ is defined by \eqref{def_oracle}. 
Let $\Theta_{\rm eb}(t)=\Theta_{\rm eb}(t,\varkappa)$ be the collection of all $\theta\in\mathbb{R}^n$ that satisfies 
the EBR condition with structural parameter $t$. \smallskip

The extent of the restriction $\theta\in\Theta_{\rm eb}(t)$ varies over different choices of constant 
$t$, becoming more lenient as $t$ increases, eventually covering the entire parameter space. 
In general, for any $t$, a sequence of $\theta$'s can be found such that 
$\theta\not \in \Theta_{\rm eb}(t)$, when $n$ varies. On the other hand, the set 
$\Theta_{\rm eb}(0)$ is not empty and consists of such $\theta$'s for which the oracle  
$I_o$ coincides with the true sparsity structure $I_0$. For a general discussion on EBR, we refer the reader 
to \cite{Belitser&Nurushev:2019}. For sparsity structure in $\ell_1$-sense, the EBR condition is satisfied if 
the minimal squared value the nonzero coordinate of $\theta$ is bounded is larger than 
a certain lower bound, but this bound depends also on the number of the nonzero coordinates of $\theta$.
For example, for some sufficiently large $C>0$,
\[
\Theta(C)=\cup_{I\in\mathcal{I}}
\big\{\theta\in L_I: |\theta_i| \ge C\log^{1/2}\big(\tfrac{en}{|I|}\big), i \in I\big\}
\subseteq \Theta_{\rm eb}(0).
\] 

 \begin{theorem} [Confidence ball]
\label{th2}
Let the confidence set $B(\cdot, \cdot)$ be defined by \eqref{def_conf_set},
$\hat{r}=p(\hat{I})$, $\hat{I}$ and $\hat{\theta}$  be given by \eqref{def_crit} 
and \eqref{def_estimator} respectively.
Then for sufficiently large $\kappa, \varkappa$ there exist constants $M_2=M_2(t), 
H_2, m_2, M_3, H_3, m_3 >0$ such that for any $t\ge 0$, 
\begin{align} 
\sup_{\theta\in\Theta_{\rm eb}(t)} \mathbb{P}_\theta
\big(\theta\notin B(\hat{\theta}, M_2(t) \hat{r})\big) 
&\le H_2 n^{-m_2}.
\label{coverage credible}  \\
\sup_{\theta\in\mathbb{R}^n}\mathbb{P}_\theta\big(\hat{r}\ge M_3 r(\theta)\big) 
&\le H_3 n^{-m_3}.
\label{size credible}
\end{align}
\end{theorem}

So far, all results are formulated in terms of the oracle. 
For estimation and the construction of a confidence set, this local approach delivers 
the most general results, as the convergence rate of the proposed estimator is directly 
linked with the oracle sparsity rather than the true structure, which may not be sparse 
(but very close to a sparse structure).
However, if the parameter has a true sparse structure, the method will attain the 
quality pertinent to that true sparsity as well.  

To illustrate this, consider any $\theta\in\mathbb{R}^n$. The true structure $I_0(\theta)$ 
and the oracle structure $I_o(\theta)$  in general do not coincide, but they are related by 
\begin{align} 
\label{oracle vs true}
r(\theta)=r(\theta,I_o)\le r(\theta,I_0)
=p(I_0)=|I_0| \log^{1/2}({en}/{|I_0|}).
\end{align}
The first two relations hold by the definition of the oracle and the third because 
$\mathrm{P}_{I_0}\theta=\theta$ for the true structure. 
Since $r(\theta)=\rho(\theta-\mathrm{P}_{I_o} \theta)+p(I_o)$, we get 
in particular that $p(I_o) \le p(I_0)$ and hence $|I_o|\le |I_0|$.
By \eqref{oracle vs true}, 
\[
\sup_{\theta\in\ell_0[s]} r(\theta) \le s\log^{1/2}({en}/{s}).
\] 
Besides, all the constants in Theorems \ref{th1} and \ref{th2} are uniform over 
$\theta\in\ell_0[s]$ and $\mathbb{P}_\theta\in \mathcal{P}_\alpha$. 
The next result follows from these facts and Theorems~\ref{th1} and \ref{th2}.

\begin{corollary}
\label{cor1}
Under the conditions of Theorems \ref{th1} and \ref{th2}, with the same choice of the constants, 
\begin{align*}
\sup_{\theta\in\ell_0[s]} \sup_{\mathbb{P}_\theta\in \mathcal{P}_\alpha} 
\mathbb{P}_{\theta}\big(\rho(\theta-\hat{\theta})
\ge M_1s\log^{1/2}({en}/{s})\big)  
&\le H_1 n^{-m_1}, \\
\sup_{\theta\in\ell_0[s]}\sup_{\mathbb{P}_\theta\in \mathcal{P}_\alpha}
\mathbb{P}_\theta\big(\hat{r} \ge M_3 s\log^{1/2}({en}/{s})\big) 
&\le H_3 n^{-m_3},
\end{align*}
and 
\eqref{coverage credible} holds. 
\end{corollary}

We claim that the obtained convergence rate $s\log^{1/2}({en}/{s})$ in terms of quintile loss 
function is optimal over the class $\ell_0[s]$ in the minimax sense. 
More precisely, the results of \cite{Johnstone&Silverman:2004} 
(in \cite{Johnstone&Silverman:2004}, relations (17) and (18) with $p=0$ and $q=1$)
and \eqref{prop_rho_abs} imply that for the normal model 
$\mathbb{P}_\theta=\bigotimes_{i\in [n]}\mathrm{N}(\theta_i,1)$, there exist absolute $C_1$ such that 
\[
\inf_{\tilde{\theta}} \sup_{\theta\in\ell_0[s]} 
\mathbb{E}_\theta|\theta-\tilde{\theta}| \ge  C_1s\log^{1/2}({en}/{s}).  
\]
This lower bound is not quite what we need to match with our upper bound because it
is formulated in terms of expectation and the $\ell_1$-loss. 
However, it is not difficult to establish the probability version of the lower bound: there exist 
absolute $C_1,C_2>0$ such that
 \[
 \inf_{\tilde{\theta}} \sup_{\theta\in\ell_0[s]} 
\mathbb{P}_\theta(|\theta-\tilde{\theta}| \ge  C_1s\log^{1/2}({en}/{s}) ) \ge C_2.
\]
The relation \eqref{prop_rho_abs} connects the $\ell_1$-loss with the quantile loss, so that 
the above relation  implies that 
\[
\inf_{\tilde{\theta}} \sup_{\theta\in\ell_0[s]} 
\mathbb{P}_\theta(\rho(\theta-\tilde{\theta}) \ge  C_1c_\tau ^{-1}s\log^{1/2}({en}/{s}) ) \ge C_2.
\]
In this form, the lower bound matches our upper bound given by Corollary \ref{cor1}, basically 
showing that our procedure attains the optimal rate $s\log^{1/2}({en}/{s})$ 
for the sparsity class  $\ell_0[s]$. This also means that the size of the constructed confidence ball 
$B(\hat{\theta}, M_2(t) \hat{r})$ is also optimal over the sparsity class $\ell_0[s]$ 
in the minimax sense. However, the unavoidable price for the optimality in the size relation 
is that the coverage relation holds uniformly over $\Theta_{\rm eb}(t)$, not over $\ell_0[s]$. 

\begin{remark}
Interestingly, although there are
some `deceptive' $\theta$ in $\ell_0[s]$ that are not covered by $\Theta_{\rm eb}(t)$, 
there are also some $\theta$'s in $\Theta_{\rm eb}(t)$ which do not belong 
to the sparsity class $\ell_0[s]$, but for which the coverage relation holds.
\end{remark}

\begin{remark}
We derived the adaptive (sparsity $s$ is also unknown) minimax results over the traditional 
sparsity scale $\{\ell_0[s], s \in [n]\}$ as consequence of our local oracle results.  
The scope of our local result is even broader, the minimax results can be derived 
over any scale of classes $\{\Theta_s, s \in \mathcal{S}\}$ with the corresponding minimax rates 
$R(\Theta_s)$ as longs as, for some $c>0$,
\[
\sup_{\theta\in\Theta_s} r(\theta) \le c R(\Theta_s), \quad s \in \mathcal{S}.
\] 
\end{remark}
Indeed, if the above relation is fulfilled, we immediately obtain all the claims of Corollary \ref{cor1} 
with $\Theta_s$ instead of $\ell_0[s]$, as consequences of Theorem \ref{th1} and \ref{th2}. 
For example, it is possible to derive the minimax results also for 
the scale of the $\ell_s$-balls: $\ell_s[\eta] = \{\theta\in\mathbb{R}^n: \tfrac{1}{n} \sum_{i=1}^n 
|\theta_i|^s \le \eta^s\}$, $s\in[0,2]$, with $\eta=\eta_n\to 0$ as $n\to \infty$.

\section{Proofs}
\label{section_proofs}
In this section we present the proofs of the theorems. 

\begin{proof}[Proof of Theorem \ref{th1}]
First we introduce some notation. Define the events 
\[
E_1=\{\rho(\theta-\hat{\theta}) \ge M_1 r(\theta)\}, \quad 
E_2=\{r(\theta,\hat{I})\le A_1 r(\theta)\}, \quad 
E_3=\{\rho(\mathrm{P}_{\hat{I}}\xi) \le A_2 p(\hat{I})\}, 
\]
where the constants $M_1, A_1,A_2>0$ are to be chosen later.
We evaluate the probability of interest $\mathbb{P}_\theta(E_1)
=\mathbb{P}_\theta(\rho(\theta-\hat{\theta}) \ge M_1 r(\theta))$ as follows:
\begin{align}
\mathbb{P}_\theta(E_1) &=
\mathbb{P}_\theta(E_1 \cap E_2\cap E_3)
+\mathbb{P}_\theta(E_1 \cap E_3^c) 
+\mathbb{P}_\theta(E_1\cap E_2^c \cap E_3) \notag \\
&= T_1+T_2+T_3.
\label{T1+T2+T3}
\end{align}
We bound these three probabilities separately.

First we bound $T_1$. Using the properties of the quantile loss function $\rho$ from Remark \ref{rem2}, 
we derive that  the event $E_1\cap E_2\cap E_3$ implies that 
\begin{align*}
M_1 r(\theta) &\leq \rho(\theta-\hat{\theta})=\rho(\theta-P_{\hat{I}} X) \\
&\le \rho(\theta-P_{\hat{I}} \theta) + \rho(-P_{\hat{I}} \xi)  
\le \rho(\theta-P_{\hat{I}} \theta)  + A_2 C_\tau  p(\hat{I}) \\
&\le  \max\{1,\tfrac{A_2C_{\tau}}{\varkappa}\} r( \theta, \hat{I}) 
\le \max\{1,\tfrac{A_2C_{\tau}}{\varkappa}\} A_1 r(\theta).
\end{align*}
Hence, if $M_1> A_1 \max\{1,\tfrac{A_2C_{\tau}}{\varkappa}\}$, then 
\begin{align}
\label{rel_E1}
T_1=\mathbb{P}_\theta(E_1 \cap E_2) 
\le\mathbb{P}_\theta\big(M_1r(\theta) \le \max\{1,\tfrac{A_2C_{\tau}}{\varkappa}\} A_1 r(\theta)\big) = 0.
\end{align}

To bound $T_2$, write 
\begin{align*}
T_2& =\mathbb{P}_\theta(E_1 \cap E_3^c)
\le \mathbb{P}_\theta(E_3^c)  
\le 
\sum \nolimits_{I \in \mathcal{I}}\mathbb{P}\big(\rho(\mathrm{P}_{{I}} \xi)> A_2 p(I) \big).
\end{align*}
If $A_2 > M_\xi$, using the last relation and Condition C1 \eqref{C1}, we obtain that 
\begin{align*}
\mathbb{P}\big(\rho(\mathrm{P}_{{I}} \xi)> A_2 p(I) \big) & =
\mathbb{P}\big(\rho(\mathrm{P}_{{I}} \xi)>M_\xi p({I}) + 
|I|^{1/2} (A_2 - M_\xi) \big(\lambda(|I|)\big)^{1/2} \big) \\
&\le H_{\xi} e^{-A_3\lambda(|I|)}, \qquad  I\in\mathcal{I},
\end{align*}
where $A_3=\alpha_\xi (A_2 - M_\xi)^2$.
Recall that $\lambda(s)$ is increasing for $s\in(0,n]$. This, \eqref{rel_nu} and the two 
previous displays entail that, for $A_3>1$ (so that $A_4= \tfrac{(A_3-1)}{2} >0$) 
\begin{align}
T_2 &\le
\mathbb{P}_\theta(E_3^c) 
\le H_{\xi}\sum_{I \in \mathcal{I}} e^{-A_3 \lambda(|I|)} \notag\\& 
\le H_{\xi}  e^{-A_4\lambda(1)}
\sum_{I \in \mathcal{I}} e^{-(1+A_4)\lambda(|I|)} \le H'_1 e^{-A_4\lambda(1)}.
\label{rel_E2}
\end{align}

Finally we bound $T_3=\mathbb{P}_\theta(E_1\cap E_2^c \cap E_3)$.   
Using \eqref{def_hat_I} with $I'=I_o$, we obtain that under $E_2^c \cap E_3$,
\begin{align*} 
C_{\tau} \rho(P_{I_0} \xi) &\ge
\rho(\mathrm{P}_{\hat{I}^c}\theta) + \kappa p(\hat{I})- \rho(\mathrm{P}_{I_o^c}\theta) 
- \kappa p(I_o) - \rho(P_{\hat{I}} \xi) \\ 
&\ge
\rho(\mathrm{P}_{\hat{I}^c}\theta)+\kappa p(\hat{I})- \rho(\mathrm{P}_{I_o^c}\theta) 
- \kappa p(I_o)-A_2 p(\hat{I}) \\ 
& \ge
\min \{\tfrac{\kappa- A_2}{\varkappa} , 1\}  r(\theta, \hat{I})   - \max \{ \tfrac{\kappa}{\varkappa} ,1 \} r(\theta)\\  
&\ge
\min \{\tfrac{\kappa- A_2}{\varkappa} , 1\}  A_1 r(\theta) - \max \{ \tfrac{\kappa}{\varkappa} ,1 \} r(\theta) \\
&=A_5 r(\theta) \ge A_5 \varkappa p(I_o),
\end{align*}
as long as $A_5=\min\{\tfrac{\kappa- A_2}{\varkappa} , 1\} A_1-\max\{\tfrac{\kappa}{\varkappa},1\}>0$. 
We conclude that the event $E_2^c \cap E_3$ implies the event $\{\rho(\mathrm{P}_{I_o}\xi) \ge 
\frac{A_5 \varkappa}{C_{\tau}} p(I_o)\}$, so that, by Condition C1 \eqref{C1}, 
\begin{align}
T_3 &\le \mathbb{P}_\theta(E_2^c \cap E_3) 
\le\mathbb{P}_\theta\big(\rho(\mathrm{P}_{I_o}\xi) \ge \tfrac{A_5\varkappa}{C_{\tau}}\, p(I_o) \big) \notag \\ 
&=\mathbb{P}_\theta\big(\rho(\mathrm{P}_{I_o}\xi) \ge M_\xi p(I_o) + 
|I_o|^{1/2}(\tfrac{A_5\varkappa}{C_{\tau}} -M_\xi) \lambda(|I_o|)^{1/2} \big) \notag \\ 
& \le H_{\xi} e^{-A_6\lambda(|I_o|)}\le  H_{\xi} e^{-A_6\lambda(1)},
\label{rel_E3}
\end{align}
where $A_6= \alpha_\xi (\frac{A_5\varkappa}{C_{\tau}}-M_\xi)^2>0$ if $A_5>C_\tau M_\xi\varkappa^{-1}$.

To summarize the choices of the constants, we need to take such constants $\kappa, M_1$ 
(in the claim of the theorem) and such constants $A_1, A_2$ (in the proof of the theorem) 
that $M_1>A_1\max\{1,\tfrac{A_2C_{\tau}}{\varkappa}\}$, 
$A_3=\alpha_\xi (A_2 - M_\xi)^2>1$, 
$\min\{\tfrac{\kappa- A_2}{\varkappa},1\} A_1-\max\{\tfrac{\kappa}{\varkappa},1\}>C_\tau M_\xi$. 
We take, for example, $A_2=M_\xi+(2/\alpha_\xi)^{1/2}$ (so that 
$A_3=\alpha_\xi (A_2-M_\xi)^2=2$, $A_4=1/2$), $\kappa=A_2+1$, $A_1=
(\tfrac{C_\tau}{\varkappa}(M_\xi+1)+\max\{\tfrac{\kappa}{\varkappa},1\})/\min(\varkappa^{-1},1)$ 
(so that $A_5=\tfrac{C_\tau}{\varkappa}(M_\xi +1)$, $A_6=\alpha_\xi$), and 
$M_1=A_1\max\{1,\tfrac{A_2C_{\tau}}{\varkappa}\}+1$. Combining \eqref{T1+T2+T3}, 
\eqref{rel_E1}, \eqref{rel_E2} and \eqref{rel_E3}, we obtain the claim of the theorem 
with the chosen $M_1$, $H_1=H'_1+H_\xi$ and $m_1=\min\{A_4,A_6\}$. 
\end{proof}

\begin{proof}[Proof of Theorem \ref{th2}]
For some fixed $\delta\in(0,1)$ (for example, take $\delta=1/2$), introduce 
the event $E_4=\{p(\hat{I})\le \delta p(I_o)\}$, where $I_o=I_o(\theta)$ is defined 
by \eqref{def_oracle}. 

First we evaluate $\mathbb{P}_\theta(p(\hat{I})\le \delta p(I_o))= \mathbb{P}_\theta(E_4)$. We have 
\begin{align*}
p(\hat{I} \cup I_o)&=|\hat{I} \cup I_o| \log^{1/2}\big(\tfrac{en}{|\hat{I} \cup I_o|}\big)\\
&\le |\hat{I}| \log^{1/2}\big(\tfrac{en}{|\hat{I} \cup I_o|}\big)+ 
|I_o|\log^{1/2}\big(\tfrac{en}{|\hat{I} \cup I_o|}\big)\\
&\le |\hat{I}| \log^{1/2}\big(\tfrac{en}{|\hat{I}|}\big)+ 
|I_o|\log^{1/2}\big(\tfrac{en}{|I_o|}\big)\\
&=p(\hat{I})+p(I_o).
\end{align*}
By using the above relation, we obtain that, under the event $E_4$, 
$p(\hat{I} \cup I_o)\le p(\hat{I})+p(I_o)\le (1+\delta)p(I_o)$. 
Hence, we have that, under $E_4$, 
\[
\tfrac{1}{1+\delta} p(\hat{I}\cup I_o) \le p(I_o) \le  p(\hat{I}\cup I_o).
\]
The last relation, \eqref{prop_monoton} with $I_1=(\hat{I} \cup I_o)^c\subseteq I_o^c = I_2$
and the definition \eqref{def_oracle} of the oracle $I_o$ imply that, under $E_4$,
\begin{align*}
 \rho(\mathrm{P}_{\hat{I}^c} \theta) - \rho(\mathrm{P}_{(\hat{I}\cup I_o)^c} \theta)
&\ge \rho(\mathrm{P}_{\hat{I}^c} \theta) - \rho(\mathrm{P}_{I_o^c} \theta)
\ge \varkappa(p(I_o)-p(\hat{I}))\notag\\
&\ge \varkappa (1-\delta)p(I_o)\ge \varkappa\tfrac{1-\delta}{1+\delta}
p(\hat{I} \cup I_o).
\end{align*}
 
Recall the event $E_3=\{\rho(\mathrm{P}_{\hat{I}}\xi) \le A_2 p(\hat{I})\}$ from 
the proof of the previous theorem and let $\kappa$ be sufficiently large to satisfy 
$\kappa>A_2$. Then, under $E_3 \cap E_4$, from \eqref{def_hat_I} with $I'=\hat{I} \cup I_o$ 
and $\kappa>A_2$, it follows that, if $\varkappa\tfrac{1-\delta}{1+\delta} -\kappa > C_\tau (M_\xi +1)$, then 
\begin{align}
C_\tau \rho(\mathrm{P}_{I_o} \xi) &\ge 
\kappa\big(p(\hat{I})-p(I')\big)- A_2 p(\hat{I}) 
+ \rho(\mathrm{P}_{\hat{I}^c} \theta) - \rho(\mathrm{P}_{I'^c} \theta) 
\label{prop_BG}\\
&\ge  \big(\varkappa\tfrac{1-\delta}{1+\delta} -\kappa\big)p(\hat{I} \cup I_o) \notag\\
&\ge  \big(\varkappa\tfrac{1-\delta}{1+\delta} -\kappa\big)p(I_o)> C_\tau(M_\xi +1) p(I_o). \notag
\end{align}
Notice that we need to assume $\varkappa$ to be sufficiently large to satisfy 
$\varkappa> 3(C_\tau M_\xi +C_\tau+\kappa)$. Using \eqref{rel_E2}, the last display 
and Condition C1 \eqref{C1}, we bound $\mathbb{P}_\theta(p(\hat{I})\le \delta p(I_o))$:  
\begin{align}
\mathbb{P}_\theta(p(\hat{I})\le \delta p(I_o)) &=
\mathbb{P}_\theta(E_4) \le \mathbb{P}_\theta(E_3^c)+
\mathbb{P}_\theta(E_3\cap E_4) \notag\\
&\le H'_1 e^{-A_4\lambda(1)} + \mathbb{P}_\theta\big(\rho(\mathrm{P}_{I_o} \xi)> (M_\xi +1) p(I_o)\big)\notag\\
&\le H'_1 e^{-A_4\lambda(1)}  + H_\xi e^{-\alpha_\xi \lambda(1)}.
\label{ineq_hatG}
\end{align}

Now we establish the coverage property. The constants $M_1$, $H_1$ and $m_1$ 
are defined in Theorem \ref{th1}. Take $M_2=\tfrac{M_1(t+\varkappa)}{\delta}$, where 
fixed $\delta\in(0,1)$ is from the definition of the event $\hat{\mathcal{G}}$. 
If $\theta\in\Theta_{\rm eb}(t)$, then, in view of \eqref{def_oracle},  
$r(\theta) = \rho(\mathrm{P}_{I_o^c}\theta) + \varkappa p(I_o) \le (t+\varkappa) p(I_o)$.
So, $p(I_o) \ge (t+\varkappa)^{-1} r(\theta)$ for all $\theta\in\Theta_{\rm eb}(t)$.
Combining this with Theorem \ref{th1} and \eqref{ineq_hatG} yields that, 
uniformly in $\theta\in\Theta_{\rm eb}(t)$,
\begin{align*}
\mathbb{P}_\theta \big(\theta\notin &B(\hat{\theta},M_2\hat{r})\big) \\
&\le \mathbb{P}_\theta\big(\rho(\theta-\hat{\theta})> M_2\hat{r},\hat{r}\ge \delta p(I_o)\big)
 +\mathbb{P}_\theta\big(\hat{r}<\delta p(I_o) \big)\notag\\
&\le \mathbb{P}_\theta\big(\rho(\theta-\hat{\theta})> M_1(t+\varkappa) p(I_o) \big)
+ \mathbb{P}_\theta\big(p(\hat{I})< \delta p(I_o)\big) \\
&\le  \mathbb{P}_\theta\big(\rho(\theta-\hat{\theta})> M_1 r(\theta)\big)
+ \mathbb{P}_\theta\big(p(\hat{I})< \delta p(I_o)\big) \\
&\le H_1 e^{-m_1 \lambda(1)} + H'_1 e^{-A_4\lambda(1)}  + H_\xi e^{-\alpha_\xi \lambda(1)}.
\end{align*}
The coverage relation follows.

Let us show the size property. 
Recall again the event $E_3=\{\rho(\mathrm{P}_{\hat{I}}\xi) \le A_2 p(\hat{I})\}$ from 
the proof of the previous theorem. By using \eqref{prop_BG} with $I'=I_o$, we derive that 
the event $\{p(\hat{I}) \ge M_3 r(\theta)\} \cap E_3$ implies the event 
$\{C_\tau \rho(\mathrm{P}_{I_o} \xi) \ge A_7p(\hat{I})\ge A_7 M_3 r(\theta)\ge 
A_7 M_3\varkappa p(I_o)\}$ with $A_7 =\kappa - A_2-M_3^{-1}>0$, where $A_2$ is defined 
in the proof of Theorem \ref{th1}. We thus have that, for any $\theta\in\mathbb{R}^n$,
\begin{align*}
\mathbb{P}_\theta(\hat{r} \ge M_3 r(\theta))
&\le
\mathbb{P}_\theta(E_3^c)+
\mathbb{P}_\theta(\hat{r} \ge M_3 r(\theta),E_3) \\
&=
\mathbb{P}_\theta(E_3^c)+\mathbb{P}_\theta(p(\hat{I}) \ge M_3 r(\theta),E_3)\\ 
&\le \mathbb{P}_\theta(E_3^c)+\mathbb{P}_\theta\big(p(\hat{I}) \ge M_3 r(\theta),
C_\tau \rho(\mathrm{P}_{I_o} \xi) \ge A_7 p(\hat{I})\big) \\
&\le \mathbb{P}_\theta(E_3^c)+\mathbb{P}_\theta\big(C_\tau \rho(\mathrm{P}_{I_o} \xi) \ge A_7p(\hat{I})
\ge A_7 M_3\varkappa p(I_o) \big) \\
& \le H'_1 e^{-A_4\lambda(1)}  + H_\xi e^{-\alpha_\xi \lambda(1)}, 
\end{align*}
where the last inequality in the above display is obtained by \eqref{rel_E2} and Condition C1 \eqref{C1}, and 
$M_3$ is chosen to be so large that $A_7 M_3\varkappa > C_\tau(M_\xi +1)$.
The size  relation follows.
\end{proof}

\newpage




\end{document}